\documentclass[12pt, reqno]{amsart}
\usepackage{amsmath}
\usepackage{amsthm}
\usepackage{amssymb}
\usepackage{amsrefs}
\usepackage{mathrsfs}
\usepackage[usenames]{color}
\usepackage{bm}
\usepackage{bbm}
\usepackage{enumitem}
 \newtheorem{thm}{}[section]
 \newtheorem{theorem}[thm]{Theorem}
 \newtheorem{corollary}[thm]{Corollary}
 \newtheorem{lemma}[thm]{Lemma}
 \newtheorem{proposition}[thm]{Proposition}
  
  \theoremstyle{definition}
 \newtheorem{definition}[thm]{Definition}
 \theoremstyle{remark}
 \newtheorem{remark}[thm]{Remark}
 
 \newtheorem{question}[thm]{Question}
 \newtheorem{example}[thm]{Example}
 \numberwithin{equation}{section}
 \allowdisplaybreaks
\newcommand{\abs}[1]{\left\lvert#1\right\rvert}
\newcommand{\norm}[1]{\left\lVert#1\right\rVert}
\newcommand{\enbrace}[1]{\left\lbrace#1\right\rbrace}
\newcommand{\enpar}[1]{\left(#1\right)}

\newcommand{\NN}{\ensuremath{\mathbb{N}}}
\newcommand{\FF}{\ensuremath{\mathbb{F}}}
\newcommand{\Ft}{\ensuremath{\mathcal{F}}}
\newcommand{\EB}{\ensuremath{\mathcal{E}}}
\newcommand{\EE}{\ensuremath{\mathbb{E}}}
\newcommand{\BB}{\ensuremath{\mathcal{B}}}
\newcommand{\GG}{\ensuremath{\mathcal{G}}}

\newcommand{\ee}{\ensuremath{\mathbf{e}}}

\newcommand{\Ind}{\ensuremath{\mathbbm{1}}}

\newcommand{\bphi}{\ensuremath{\mathbf{\Phi}}}
\newcommand{\RR}{\ensuremath{\mathbb{R}}}
\newcommand{\supp}{\operatorname{supp}}
\newcommand{\CC}{\ensuremath{\mathbb{C}}}

\newcommand{\Ks}{\ensuremath{\bm{K}}}
\newcommand{\DL}{\ensuremath{\bm{D}}}
\newcommand\ww{\ensuremath{{\bm{w}}}}
\newcommand{\XX}{\ensuremath{\mathbb{X}}}

\newcommand{\xx}{\ensuremath{\bm{x}}}
\DeclareMathOperator{\sgn}{sign}
\DeclareMathOperator{\Ave}{Ave}

\begin{document}
\title[Counterexamples  in isometric theory of greedy bases]{Counterexamples in isometric theory of symmetric and greedy bases}
\author[F. Albiac]{Fernando Albiac}
\address{Department of Mathematics, Statistics, and Computer Sciencies--InaMat2 \\
Universidad P\'ublica de Navarra\\
Campus de Arrosad\'{i}a\\
Pamplona\\
31006 Spain}
\email{fernando.albiac@unavarra.es}

\author[J. L. Ansorena]{Jos\'e L. Ansorena}
\address{Department of Mathematics and Computer Sciences\\
Universidad de La Rioja\\
Logro\~no\\
26004 Spain}
\email{joseluis.ansorena@unirioja.es}

\author[Ó. Blasco]{Óscar Blasco}
\address{Department of Mathematical Analysis\\
Universidad de Valencia\\
Burjassot\\
46100 Spain}
\email{oscar.blasco@uv.es}

\author[H. V. Chu]{H\`{u}ng Vi\d{\^ e}t Chu}
\address{Department of Mathematics\\
University of Illinois Urbana-Champaign \\ Urbana\\ IL 61820\\USA}
\email{hungchu2@illinois.edu}

\author[T. Oikhberg]{Timur Oikhberg}
\address{Department of Mathematics\\
University of Illinois Urbana-Champaign \\ Urbana\\ IL 61820\\USA}
\email{oikhberg@illinois.edu}

\subjclass[20120]{41A65, 46B15, 46B45}
\begin{abstract}
We continue the study initiated in \cite{AW2006} of  properties related to greedy bases in the case when the constants involved are  sharp, i.e., in the case when they are equal to $1$.  Our main goal here is to provide  an example of a Banach space with a basis that satisfies  Property (A) but fails to be $1$-suppression unconditional, thus settling Problem 4.4 from  \cite{AA}. 
In particular, our construction demonstrates that bases with Property (A) need not be $1$-greedy even with the additional assumption that they are unconditional and symmetric.
We also exhibit a finite-dimensional counterpart of this example, and  show that, at least in the finite-dimensional setting, Property (A) does not pass to the dual. As a by-product of our arguments, we prove that a symmetric basis is unconditional if and only if it is total, thus generalizing the well-known result that  symmetric Schauder bases are unconditional.
\end{abstract}
\thanks{F. Albiac acknowledges the support of the Spanish Ministry for Science and Innovation under Grant PID2019-107701GB-I00 for \emph{Operators, lattices, and structure of Banach spaces}. The University of Illinois partially supported the work of H. V. Chu and T. Oikhberg via Campus Research Board award 23026.}
\maketitle
\section{Introduction}\noindent
The roots of this paper lie in the analysis of the optimality of the thresholding greedy algorithm relative to bases in Banach spaces. This optimality is reflected in the sharpness of the constants that appear in the definitions of the different types of greedy-like bases. What justifies studying the ``isometric'' case in general is the fact that various approximation algorithms converge trivially when some appropriate constant is 1. By contrast, when this constant is strictly greater than 1, the problem of convergence can be very difficult to resolve. We mention as an example the so-called $X$-greedy algorithm, whose convergence for the normalized Haar basis of $L_{p}[0,1]$ (whose unconditional constant is strictly greater than 1) is still an open problem.  The reader is referred to \cites{DOSZ2010,Livshits2010} for details and recent developments on this matter. 

For the sake of self-reference we next outline the most relevant notions; the reader is referred to \cite{AlbiacKalton2016}*{Section 10} or to the recent monograph \cite{AABW} for more details.

 Let $\BB=(\xx_n)_{n=1}^{\infty}$ be a family of vectors in a Banach space $\XX$ over the field $\FF=\RR$ or $\CC$. 
Suppose that 
\begin{enumerate}[label=(B.\arabic*),leftmargin=*]
\item\label{complete} $\BB$ is \emph{fundamental}, i.e., the closure of the span of $\BB$ is the whole space $\XX$, 
\item\label{biorthogonal}  there is a (unique) family  $\BB^* = (\xx_n^*)_{n=1}^{\infty}$ in $\XX^*$ biorthogonal to $\BB$, and 
\item\label{seminormalized} $\sup_n \enbrace{ \norm{\xx_n},  \norm{\xx_n^*}}<\infty$.
\end{enumerate}

Throughout this paper such a family $\BB$ will be called a \emph{basis} of $\XX$. 
Sequences $\BB$ that satisfy \ref{complete},  \ref{biorthogonal}, and are \emph{total}, i.e., the span of $\BB^*$ is weak$^*$-dense in $\XX^*$, are called \emph{Markushevich bases}. Schauder bases are a particular case of  Markushevich bases; however we emphasize that we do not  impose a priori  totality to a basis  unless we say otherwise.

A \emph{greedy ordering}  for a vector $x\in\XX$ with respect to a basis $\BB$ is an injective map 
\[
\rho\colon \NN \to \NN
\]
whose image contains the \emph{support} of $x$, 
\[
\supp(x)=\{n\in\NN \colon \xx_{n}^{\ast}(x)\not=0\},
\]
and satisfies $\vert\xx_{\rho(i)}^{\ast}(x)\vert\ge\vert\xx_{\rho(j)}^{\ast}(x)\vert$ if $i\le j$. 
If the sequence $(\xx_n^*(x))_{n=1}^{\infty}$ contains several terms with the same absolute value then $x$ has multiple greedy orderings.  To obtain uniqueness we impose on $\rho$ the extra assumption that if $i<j$, then either 
$\vert \xx^*_{\rho(i)}(x)\vert > \vert \xx^*_{\rho(j)}(x)\vert$ or $\vert\xx^*_{\rho(i)}(x)\vert = \vert \xx^*_{\rho(j)}(x)\vert$ and $\rho(i) < \rho(j)$.
We will refer to this ordering  as the \emph{natural greedy ordering} for $x$. With this convention, the \emph{$m$-th natural greedy approximation}  to $x$ is given  by
\begin{equation}\label{eqTGA}
\GG_{m}(x) = \sum_{j=1}^{m}\xx_{\rho(j)}^{\ast}(x)\xx_{\rho(j)}, \quad m\in\NN,
\end{equation}
where $\rho$ is the natural greedy ordering. The \emph{thresholding greedy algorithm}, or TGA for short, is the sequence of  operators $(\GG_m)_{m=1}^{\infty}$  given by the formula in \eqref{eqTGA}.

In order to guarantee the reconstruction of ``signals'' in $\XX$, and thus make the TGA a reasonable method of approximation, we impose the mild requirement that $(\GG_{m}(x))_{m=1}^{\infty}$ converge to $x$ for every $x\in \XX$. 
Wojtaszczyk \cite{Woj2000} (cf.\ \cite{AABW}*{Theorem 4.1}) proved that this condition is equivalent to the uniform boundedness of the (nonlinear and unbounded) operators $\GG_m$, i.e,
\begin{equation}\label{eqdefQG}
\norm{ \GG_{m}(x)}\le C\norm{ x},\quad x\in \XX, \, m\in\NN,
\end{equation}
for some constant $C\ge 1$. Following \cite{KT99}, we call such bases \emph{$C$-quasi-greedy}.

To gauge the performance of the TGA it is convenient to compare the accuracy of the greedy approximation with the best possible one. Given $x \in \XX$ and $m \in \NN$ we consider two different theoretical approximation errors,
\[
\sigma_m(x) = \inf \left\{ \norm{ x - \sum_{n \in A} \alpha_n \, \xx_n } \colon \alpha_n \in \FF, \, \abs{A} \leq m \right\} ,
\]
and
\[
\widetilde{\sigma}_m(x) = \inf \left\{ \norm{ x - S_A(x)} \colon \abs{A} \leq m \right\},
\]
 where $S_A(x)=\sum_{n\in A} \xx_n^*(x)\, \xx_n$ is the natural projection onto the linear span of $\{\xx_{n}\colon n\in A\}$. Konyagin and Temlyakov  \cite{KT99} defined a basis to be \emph{greedy} if $\GG_{m}(x)$ is essentially the \emph{best $m$-term approximation} to $x$ using basis vectors, i.e., there exists a constant $C\ge 1$ such that
\begin{equation}\label{defgreedy}
\norm{ x-\GG_{m}(x)} \le C \sigma_m(x),\quad x\in \XX,\; m\in\NN.
\end{equation}
The smallest constant $C$ in \eqref{defgreedy} is the \emph{greedy constant} of the basis and will be denoted by $C_{g}$. If $C_g\le C$, we say that $\BB$ is $C$-greedy. Note that if  $C_g=1$, then $\norm{ x-\GG_{m}(x)} =\sigma_{m}(x)$ for all $x\in \XX$ and $m\in\NN$, so the greedy algorithm gives the {\emph best} $m$-term approximation for each $x\in \XX$.

Konyagin and Temlyakov \cite{KT99} showed that greedy bases can be intrinsically characterized as unconditional bases with the additional property of being \emph{democratic}, i.e., there exists a constant $\Delta\ge 1$ so that
\[
\norm{\Ind_A} \le \Delta \norm{\Ind_B}.
\]
whenever $\abs{A}=\abs{B}$. Here, as  is customary, we use the notation
\[
\Ind_A=\Ind_A[\BB,\XX]=\sum_{n\in A} \xx_{n}.
\] 
Vectors whose coefficients have modulus one are also relevant in greedy approximation. Given $\varepsilon=(\varepsilon_n)_{n\in A}$ in $\EE=\{\lambda\in\FF \colon \abs{\lambda}=1\}$, we set
\[
\Ind_{\varepsilon,A}=\Ind_{\varepsilon,A}[\BB,\XX]=\sum_{n\in A} \varepsilon_n\, \xx_{n}.
\] 

Recall also  that a basis  $(\xx_n)_{n=1}^\infty$  is \emph{unconditional}   if for $x\in \XX$, the series $\sum_{i=1}^\infty \xx_{\pi(i)}^*(x)\xx_{\pi(i)}$ converges  to $x$ for any permutation $\pi$ of $\NN$. It is well known that the property of being unconditional is equivalent to that of being \emph{suppression unconditional}, which means that there is a constant $K\ge 1$ such that 
\begin{equation}\label{unifbound}
\norm{ S_A(x)} \le K \norm{ x}, \quad A\subset\NN, \, \abs{A}<\infty, \, x\in\XX.
\end{equation}
The smallest constant $K$ in \eqref{unifbound} is called the \emph{suppression unconditional constant} of the basis and is denoted by $K_{s}$. If  \eqref{unifbound} holds for a  constant $K$ we say that $\BB$ is \emph{$K$-suppression unconditional.}   If $(\xx_n)_{n=1}^\infty$ is unconditional then it is \emph{$C$-lattice unconditional}, that is, there is a constant $C\ge 1$ so that
\begin{equation}\label{latticebound}
\norm{\sum_{n=1}^\infty b_n\,\xx_n} \le C \norm{\sum_{n=1}^\infty a_n\,\xx_n}, \quad (a_n)_{n=1}^\infty\in c_{00},\,  \abs{b_n}\le \abs{a_n}.
\end{equation}
We denote by $K_l$ the optimal constant $C$ in \eqref{latticebound}.  
 
Dilworth et al.\  \cite{DKKT2003} relaxed the condition defining greedy bases by comparing, for each $x$ and $m$, the error in the approximation of $x$ by $\GG_{m}(x)$ with the best theoretical approximation error from projections, $\widetilde{\sigma}_{m}(x)$. They   defined a basis to be $C$-\emph{almost greedy}, $1\le C <\infty$, if 
\begin{equation}\label{defalmostgreedy}
\norm{ x-\GG_{m}(x)} \le C \widetilde{\sigma}_m(x),\quad x\in \XX,\; m\in\NN.
\end{equation}

Going back to the optimality issues that  are our concern  in this note, it is  straightforward  to check that if $\BB$ is an orthonormal basis of a Hilbert space then $\BB$ is $1$-greedy.  However, due to computational issues, the verification of  condition \eqref{defgreedy} can be very hard even for well-known bases in Banach spaces.   Answering a question raised by Wojtaszczyk  in \cite{Wojt2003},  the authors  found in \cite{AW2006}  the following characterization of $1$-greedy bases in the spirit of the aforementioned characterization of greedy bases by Konyagin and Temlyakov. 

\begin{theorem}[\cite{AW2006}*{Theorem 3.4}]\label{AW}
A basis  of a Banach space $\XX$ is $1$-greedy if and only if it is $1$-suppression unconditional and satisfies Property (A).
\end{theorem}

Property (A) is a weak symmetry condition that was introduced in  \cite{AW2006} and that was generalized to the concept of  symmetry for largest coefficients in \cite{AA}*{Definition 3.1}. A basis $\BB$ is \emph{symmetric for largest coefficients} if there is a constant $C\ge 1$ such that
\begin{equation}\label{eq:SLC}
\norm{\Ind_{\varepsilon,A}+f} \le C \norm{\Ind_{\varepsilon,B}+f}
\end{equation}
for all $f\in \XX $ with $\max_n \abs{\xx_n^*(f)}\le 1$, all $A$, $B\subset \NN$ with $\abs{A}\le \abs{B}<\infty$ and $A\cap B = A\cap\supp(f)=B\cap\supp(f)=\emptyset$, and all $\varepsilon\in\EE^{A\cup B}$. 
If we can choose $C=1$ in \eqref{eq:SLC} we say that $\BB$ has   Property (A).

The neat description of $1$-greedy bases provided by Theorem~\ref{AW} inspired further work in the isometric theory of greedy bases which led to the following characterizations of $1$-quasi-greedy bases and $1$-almost greedy bases precisely in terms  of the same ingredients  but in disjoint occurrences.  

\begin{theorem}[\cite{AlbiacAnsorena2016JAT}*{Theorem 2.1}]\label{AAThm} A   basis  of a Banach space   is $1$-quasi-greedy if and only if  it  is  1-suppression unconditional.
\end{theorem}

\begin{theorem}[\cite{AA}*{Theorem 2.3}]\label{AAThm2}
A basis of a Banach space is $1$-almost greedy if and only if  it satisfies Property (A).
\end{theorem}

Since almost greedy basis are in particular quasi-greedy, one could expect that  when the almost greedy constant is sharp, the implication would still hold, i.e., that being $1$-almost greedy implies being $1$-quasi-greedy. In light of  Theorems \ref{AAThm} and \ref{AAThm2},  we arrive naturally at the following question,  which enquires about the overlapping of the two properties that characterize $1$-greedy bases. 

\begin{question}\label{Q1}
Does Property (A) imply unconditionality with $K_s = 1$ (see \cite{AA}*{Problem 4.4})?
\end{question}

Theorem~\ref{AAThm} is somewhat surprising since it connects nonlinear properties in approximation theory in Banach spaces (quasi-greediness) with  linear properties (such as unconditionality). Besides, it exhibits how an isometric property could lead to an improvement of the qualitative behaviour of a basis. It is therefore natural to ask whether this is also the case with almost greedy bases.
 
\begin{question}\label{Q2} Is there a conditional basis with Property (A) (see \cite{AA}*{Problem 4.4})?
\end{question}

If we restrict our attention to unconditional bases, the study of the relation between Property (A) and unconditionality reduces to the problem of determining whether Property (A) implies some upper bound for the unconditionality constant of the basis. 

\begin{question}\label{Q3} 
Does an unconditional basis with Property (A) always have $K_s = 1$?
\end{question}

Observe that a negative answer to either Question~\ref{Q2} or  \ref{Q3} is also a negative answer to Question~\ref{Q1}.  In this note, we answer Questions~\ref{Q1} and \ref{Q3} negatively by renorming the space $\ell_1$ so that the standard unit vector basis still satisfies Property (A) and is unconditional with $K_s > 1$.

\begin{theorem}[Main Theorem]\label{thm:Main}
There exists a basis $\BB$  equivalent to the canonical basis of $\ell_1$ which satisfies  Property (A) but fails to be $1$-suppression unconditional. 
Moreover, $\BB$ is $1$-symmetric  and $1$-bidemocratic.
\end{theorem}

We observe that in our context  a \emph{$C$-symmetric basis}, $1\le C<\infty$,   means a basis $(\xx_n)_{n=1}^\infty$ such that
\[
\norm{ \sum_{n=1}^\infty  a_n \,  \xx_n} \leq C \norm{ \sum_{n=1}^\infty  a_n \,  \xx_{\pi(n)}} 
\]
 for any sequence of scalars $(a_n)_{n=1}^\infty \in c_{00}$ and any permutation $\pi$ of $\NN$. 
 Recall that a basis is said to be \emph{$C$-bidemocratic} if
\[
 \norm{\Ind_{\varepsilon,A}[\BB,\XX]}  \norm{\Ind_{\delta,B}[\BB^*,\XX^*]} \le C m
\]
for all $A$, $B\subset\NN$ with $\max\{\abs{A},\abs{B}\}\le m\in\NN$, all $\varepsilon\in\EE^A$, and all $\delta\in\EE^B$. Note that  $1$-bidemocracy does not imply Property (A).  Indeed, let $\ww=(w_n)_{n=1}^\infty$ be the weight defined by $\sum_{n=1}^m w_n=\sqrt{m}$ for all $m\in\NN$. Then the canonical basis of the $\ell_2$-direct sum of the Hilbert space $\ell_2$ and the Lorentz sequence space $d_1(\ww)$ (which is the space $\ell_{2,1}$ up to renorming) is clearly $1$-bidemocratic but fails to have Property (A). In fact, \cite{AABBL}*{Proposition 3.17} provides a more ``extreme'' example, namely,   an isometrically bidemocratic basis  which is not quasi-greedy. By \cite{AABW}*{Corollary 5.8}, this basis is symmetric with largest coefficients.  By Theorem \ref{AAThm2}, no renorming of $\XX$ will confer Property (A) on this basis.

Combining Theorem \ref{thm:Main} with Theorem~\ref{AW} yields the following consequence.

\begin{corollary}
A basis with Property (A)  needs not be $1$-greedy, and a $1$-almost greedy basis needs not be $1$-quasi-greedy.
\end{corollary}

The paper is structured as follows. In Section \ref{s:building bases},  given a weight  $\ww$ we construct a space $\DL_\ww$ that will play a central role in our subsequent  arguments; we examine properties of these spaces and their canonical bases. In Section \ref{s:suppression unconditionality} we use the spaces $\DL_\ww$ for appropriate $\ww$ to construct an example that proves Theorem \ref{thm:Main}.
Specifically, we present a $1$-symmetric distortion of the canonical $\ell_1$-basis which has Property (A) but is not $1$-suppression unconditional (see Proposition~\ref{p:renormingclosetotwo}). We also discuss the possible values of the suppression unconditionality constant.

This construction is ``rigid'' in the sense that any $1$-symmetric basis which possesses Property (A) while failing $1$-suppression unconditionality must be equivalent to the $\ell_1$-basis (Proposition \ref{p:can distort l1 only}). Along the way we show that a symmetric basis is unconditional if and only if it is total (Corollary \ref{cor:symunc}).

Finally, in Section \ref{s:fin dim}, we turn our attention to the finite-dimensional setting. Modifying our approach from Section \ref{s:suppression unconditionality} we construct a basis in $\RR^d$ equipped with a certain norm which has Property (A) but fails to be $1$-suppression unconditional. 
We also provide an example evincing that  Property (A) does not dualize at least for bases in finite-dimensional spaces.
\section{Building bases with Property (A)}\label{s:building bases}\noindent

\noindent In this section we deal with Banach spaces $\XX$ for which the unit vector system $\EB=(\ee_j)_{j=1}^\infty$ of $\FF^\NN$ is a basis. This means that $\XX\subset\FF^\NN$, $c_{00}$ is dense subspace of $\XX$, the coordinate functionals $\EB^*=(\ee_j^*)_{j=1}^\infty$, defined for each $j\in\NN$ by $\ee_j^*(f)=a_j$ for all  $f=(a_j)_{j=1}^\infty\in c_{00}$, restrict to functionals of $\XX$, and 
\[
\sup_j\{ \max\{\norm{\ee_j}_\XX,\norm{\ee_j^*}_{\XX^*}\}\}<\infty.
\]
Given such a space $\XX$ and  $A\subset\NN$ finite, the coordinate projection on $A$ with respect to $\EB$ is the extension to $\XX$ of the map
\[
S_A \colon c_{00} \to c_{00}, \quad  f \mapsto f \, \chi_A,
\]
where $\chi_A$ is the characteristic function of $A$. Given $\varepsilon\in\EE^A$, we put
\[
\Ind_{\varepsilon, A}=\Ind_{\varepsilon, A}[\EB,\XX], \quad \Ind^*_{\varepsilon, A}=\Ind_{\varepsilon, A}[\EB^*,\XX^*.]
\]

A weight will be  a nonincreasing sequence $\ww=(w_j)_{j=1}^\infty$ of positive scalars with $w_1 = 1$. We do not assume that $\lim_j w_j = 0$. Quite the contrary, examples with
\[
\ww_\infty:=\lim_j w_j>0.
\]
are essential to us. We shall use the notation 
\[
s_n = \sum_{j=1}^n w_j, \; n\in\NN,\]for the so-called \emph{primitive weight} of $\ww$. Note that
\[
w_j \leq \frac{s_j}{j}, \quad j\in\NN,
\] and 
\[\lim_n \frac{s_n}{n}=\ww_\infty.
\]

Given $E\subset\NN$ with $\abs{E}=n<\infty$, we denote by $\Pi(E)$ the set consisting of all bijections from the set 
\[
\NN_{>n}:=\{j\in\NN \colon j>n\}
\] 
onto $\NN\setminus E$. For $f=(a_j)_{j=1}^\infty\in c_{00}$, we will use the following notation
\begin{align*}
\bphi_1(E;f) &=\frac{s_n}{n} \sum_{j\in E} \abs{a_j},\\
\bphi_2(E;f)&=\sup_{\varphi\in \Pi(E)} \abs{\sum_{j=n+1}^\infty a_{\varphi(j)} w_j},\\
\bphi(E;f)&=\bphi_1(E;f)+\bphi_2(E;f),
\end{align*}
with the conventions that $s_0=0$, $s_0/0=\infty$, that any sum over an empty set is null, and that $0\cdot\infty=0$. In particular, $\bphi_1(\emptyset;f)=0$.

Now, given a weight $\ww=(w_j)_{j=1}^\infty$, we define
\[
\norm{f}_{\DL,\ww}=\sup_{\substack{ E \subset\NN \\  \abs{E}<\infty}} \bphi(E;f).
\]

We point out that, as far as the isomorphic theory is concerned, the way we combine $\bphi_1$ and $\bphi_2$ to obtain  $\bphi$ does not matter. However, using the $\ell_1$-norm is essential from the isometric point of view we will develop below.

Given a permutation  $\pi\in\Pi:=\Pi(\emptyset)$ of $\NN$ and  $f=(a_j)_{j=1}^\infty$,  we put $f_\pi=(a_{\pi(j)})_{j=1}^\infty$. If $f\in c_0$ there is a unique nonincreasing  sequence  $g$ such that $g_\pi=\abs{f}$ for a suitable  $\pi\in\Pi$. We call $g$ the  \emph{nonincreasing rearrangement} of $f$. 

Let us consider the \emph{Marcinkiewicz norm} associated with a weight $\ww$,
\[
\norm{f}_{m,\ww}=\sup_{\substack{ E \subset\NN \\  \abs{E}<\infty}} \bphi_1(E;f)=\sup_{n\in\NN} \frac{s_n}{n} \sum_{j=1}^n b_j,
\]
and the \emph{Lorentz norm}
\[
\norm{f}_{1,\ww}=\sum_{j=1}^\infty b_j \, w_j,
\]
where $(b_j)_{j=1}^\infty$ is the nonincreasing rearrangement of $f\in c_0$. Notice that 
\begin{equation}\label{eq:Marc}
\norm{\Ind_{\varepsilon, A}}_{m,\ww}=\norm{\Ind_{\varepsilon, A}}_{1,\ww}=s_n, \quad \abs{A}=n<\infty,\, \varepsilon\in\EE^A.
\end{equation}

The Banach space consisting of all $f\in c_0$ such that $\norm{f}_{1,\ww}<\infty$ will be denoted by $d_1(\ww)$. 

We start by recording some  properties of  $\norm{ \cdot}_{\DL,\ww}$. 

\begin{lemma}\label{lem:PropNorm} Let $\ww$ be a weight.
\begin{enumerate}[label=(\roman*),leftmargin=*,widest=xiii]
\item\label{P:1} $\norm{f+g}_{\DL,\ww} \le \norm{f}_{\DL,\ww}+ \norm{g}_{\DL,\ww}$ for all $f$, $g\in c_{00}$.
\item\label{P:2} $\norm{\lambda f}_{\DL,\ww} =\abs{\lambda} \norm{f}_{\DL,\ww}$ for all $\lambda\in \FF$ and $f\in c_{00}$.
\item\label{P:13} For all $f\in c_{00}$, all $E\subset\NN$ finite, and all $\pi\in\Pi$,
\[
\bphi_1(E;f_\pi)=\bphi_1(\pi(E);f), \quad \bphi_2(E;f_\pi)=\bphi_2(\pi(E);f).
\]
\item\label{P:0} $\norm{f_\pi}_{\DL,\ww}=\norm{f}_{\DL,\ww}$ for all $f\in c_{00}$ and all $\pi\in\Pi$. 
\item\label{P:8} If $\abs{f}\le g\in c_{00}$ then $\norm{f}_{\DL,\ww}\le \norm{g}_{\DL,\ww}$.
\item\label{P:9} If $f\in c_{00}$ is nonnegative and $\abs{E}=n$ then 
\[
\bphi_2(E;f)= \sum_{k=1}^\infty b_k w_{k+n},
\]
where $(b_j)_{j=1}^\infty$ is the nonincreasing rearrangement of $S_{\NN\setminus E}(f)$.
\item\label{P:16} If the sign of the components of $f=(a_j)_{j=1}^\infty\in c_{00}$ is constant on the complement of a set $E$ and $F\subset\NN\setminus E$ is a greedy set of $S_{\NN\setminus E}(f)$ then
\[
\bphi(D;f) \le  \bphi(E\cup F;f)
\]
for every set $E\subset D\subset\NN$ with $\abs{D}=\abs{E}+\abs{F}$.

\item\label{P:15} If the sign of the components of $f=(a_j)_{j=1}^\infty$ is constant on the complement of a greedy set $E$,  then
\[
\bphi(D;f) \le  \bphi(E;f)
\]
for every finite set $D\supseteq E$.
\item\label{P:14}  $ \norm{f}_{\DL,\ww}=  \norm{f}_{1,\ww}$ for all $f\in c_{00}$ having constant sign.
\item\label{P:12}  $\norm{f}_{m,\ww}\le  \norm{f}_{\DL,\ww}\le  \norm{f}_{1,\ww}$ for all $f\in c_{00}$.
\item\label{P:6} $\norm{\Ind_{\varepsilon A}}_{\DL,\ww}=s_m$ for all $A\subset\NN$ with $m=\abs{A}<\infty$ and all $\varepsilon \in \EE^A$.
\item\label{P:7} If $f\in c_{00}$ then 
\[\abs{\Ind_{\varepsilon, A}^*(f)} \le  \norm{f}_{\DL,\ww} m/s_m\] for all $A\subset\NN$ with $m=\abs{A}<\infty$ and all $\varepsilon \in \EE^A$.
\item\label{P:4} $\norm{\ee_j}_{\DL,\ww}=1$ for all $j\in\NN$.
\item\label{P:3} $\abs{\ee_j^*(f)} \le  \norm{f}_{\DL,\ww}$ for all $j\in\NN$ and $f\in c_{00}$.
\end{enumerate}
\end{lemma}
\begin{proof}
\ref{P:1}, \ref{P:2},  \ref{P:13}, \ref{P:8}, and the left-side  of \ref{P:12} are clear; \ref{P:0} is a consequence of  \ref{P:13}, and \ref{P:9} follows from the rearrangement inequality. 

We prove \ref{P:16}. Without loss of generality we  assume that $D \cap F = \emptyset$. Otherwise we would replace $E$ and $F$ with $E' = E \cup (D \cap F)$ and $F' = F \backslash (D \cap F)$. Then $D \cap F' = \emptyset$, $E \cup F = E' \cup F'$, and $F'$ is a greedy set for $S_{\NN \backslash E'} (f)$.

Pick a bijection $\sigma\colon F \to D\setminus E$. Set $n=\abs{D}$. Given $\varphi\in \Pi(D)$, let $\psi\colon \NN_{>n} \to \NN\setminus(E\cup F)$ be the map defined by $\psi(j)=\sigma(\varphi(j))$ if $j\in G:=\varphi^{-1}(F)$ and $\psi(j)=\varphi(j)$ otherwise. Then
\begin{align*}
\abs{\sum_{j=n+1}^\infty a_{\varphi(j)} w_{j}}
&=\abs{\sum_{j=n+1}^\infty a_{\psi(j)} w_{j}+\sum_{j\in G}  \enpar{ a_{\varphi(j)} -a_{\sigma(\varphi(j))} }w_{j}}\\
&\le \abs{\sum_{j=n+1}^\infty a_{\psi(j)} w_{j}}+\sum_{j\in G}  \abs{ a_{\varphi(j)} -a_{\sigma(\varphi(j))}}w_{j}\\
&\le \abs{\sum_{j=n+1}^\infty a_{\psi(j)} w_{j}}+   \frac{s_n}{n}\sum_{j\in G} \abs{ a_{\varphi(j)} -a_{\sigma(\varphi(j))}}\\
&= \abs{\sum_{j=n+1}^\infty a_{\psi(j)} w_{j}}+   \frac{s_n}{n}\sum_{j\in F} \enpar{\abs{ a_j} -\abs{a_{\sigma(j)} }}\\
&= \abs{\sum_{j=n+1}^\infty a_{\psi(j)} w_{j}}+   \frac{s_n}{n}\enpar{\sum_{j\in F} \abs{ a_j} - \sum_{j\in D\setminus E} \abs{ a_j}  }\\
&= \abs{\sum_{j=n+1}^\infty a_{\psi(j)} w_{j}}+  \bphi_1(E\cup F;f) -\bphi_1(D;f).
\end{align*}
Since $\psi\in \Pi(E\cup F)$,
\[
\bphi_2(D;f)\le \bphi_2(E\cup F;f)+  \bphi_1(E\cup F;f) -\bphi_1(D;f),
\]
as desired.

Taking into consideration \ref{P:16}, it suffices to prove \ref{P:15} in the case when $D$ is a greedy set of $f$. To that end, by induction it suffices to consider the case when $D\setminus E$ is a singleton. Let $(b_j)_{j=1}^\infty$ be the nonincreasing rearrangement of $f$. If $n=\abs{E}$, then $(b_{n+j})_{j=1}^\infty$ is the nonincreasing rearrangement of $S_{\NN\setminus E}(f)$ and $(b_{n+j+1})_{j=1}^\infty$ is the nonincreasing rearrangement of $S_{\NN\setminus D}(f)$. By  \ref{P:9},
\begin{align*}
\bphi(E;f)&=\frac{s_n}{n} \sum_{j=1}^n b_j + \sum_{j=n+1}^\infty b_j w_j,\\
\bphi(D;f)&=\frac{s_{n+1}}{n+1} \enpar{\sum_{j=1}^{n+1} b_j}+ \sum_{j=n+2}^\infty b_j w_j.
\end{align*}
Since $s_{n+1}/(n+1)-s_n/n\le 0$ and $b_j\ge b_{n+1}$ for all $j=1$, \dots, $n$,
\begin{align*}
\bphi(D;f)-\bphi(E;f)
&=\enpar{\frac{s_{n+1}}{n+1}-\frac{s_n}{n}} \sum_{j=1}^n b_j  +b_{n+1}\enpar{\frac{s_{n+1}}{n+1}-w_{n+1}} \\
&\le b_{n+1} \enpar{n \enpar{\frac{s_{n+1}}{n+1}-\frac{s_n}{n}}  + \frac{s_{n+1}}{n+1}-w_{n+1}}=0.
\end{align*}

\ref{P:14} follows from the combination of   \ref{P:9} and \ref{P:15}. In turn, the  right-hand side inequality in  \ref{P:12} is a consequence of combining \ref{P:8}  and \ref{P:14}. In light of \eqref{eq:Marc}, \ref{P:6} is consequence of \ref{P:12}.   Since the mere definition of the norm gives
\[
\frac{s_m}{m} \abs{\Ind_{\varepsilon, A}^*(f)} \le  \norm{f}_{m,\ww}, \quad f\in c_{00}, \, \abs{A} =m,
\]
\ref{P:7} also follows from \ref{P:12}.
Finally, \ref{P:4} is a particular case of \ref{P:6}, while \ref{P:3} is a particular case of \ref{P:7}.
\end{proof}

Combining Lemma~\ref{lem:PropNorm}~\ref{P:1}, \ref{P:2}, \ref{P:6} and \ref{P:3} gives that 
$(c_{00}, \norm{\cdot}_{\DL,\ww})$ is a normed space, and the unit vector system is a basis of its completion.



\begin{definition}
We define $\DL_\ww$ as the sequence space we obtain by completing $(c_{00}, \norm{\cdot}_{\DL,\ww})$.
\end{definition}

Given $f=(a_j)_{j=1}^\infty\in c_{00}$, put
\[
E_f=\{j\in \NN \colon \abs{a_j}=\norm{f}_\infty\}.
\]
\begin{lemma}\label{lem:PA}
Let $\ww$ be a weight and $f=(a_j)_{j=1}^\infty\in c_{00}$. Then 
\[
\norm{f}_{\DL,\ww}=\sup_{\substack{ E_f\subset E \subset\NN \\  \abs{E}<\infty}} \bphi_1(E;f) + \bphi_2(E;f).
\] 
\end{lemma}

\begin{proof}
Without loss of generality we assume that $\norm{f}_\infty=1$. It suffices to prove that if $E\subset \NN$ satisfies $n=\abs{E}<\infty$ and $k\in E_f\setminus E$, then
\begin{equation}\label{e0}\bphi_1(E;f) + \bphi_2(E;f)\le  \bphi_1(E\cup\{k\};f) + \bphi_2(E\cup\{k\};f).\end{equation}
 To that end, we pick $\varphi\in\Pi(E)$. Let $p\in\NN_{>n}$ be such that $\varphi(p)=k$. Let  $\psi\in\Pi(E \cup\{k\}) $ be the map given by
 \[
 \psi(j)= \begin{cases} \varphi(j) & \mbox{ if }  j \not= p,  \\  \varphi(n+1) & \mbox{ if }  j=p. \end{cases}
 \]
 We have 
 \begin{align*}
 \abs{ \sum_{j=n+1}^\infty a_{\varphi(j)} w_{j} -  \sum_{j=n+2}^\infty a_{\psi(j)} w_{j}} &=\abs{(w_{n+1}-w_p)a_{\varphi(n+1)} + w_p a_{\varphi(p)}}\\
& \le w_{n+1}-w_p +w_p\\
&=w_{n+1}.
 \end{align*}
 Hence, 
 \begin{equation}\label{e1}\bphi_2(E;f)\le  \bphi_2(E\cup\{k\};f)+w_{n+1}.\end{equation}
 In turn, since the map
\[
u\mapsto \frac{s_{n+1}}{n+1}(1+u)- \frac{s_n}{n}u, \quad u\ge 0,
\]
is nonincreasing, and $S:=\sum_{j\in E} \abs{a_j} \le n$, we have
 \begin{align}\label{e2}
\bphi_1(E\cup\{k\};f) -\bphi_1(E;f)
&= \frac{s_{n+1}}{n+1}(1+S)- \frac{s_n}{n}S \nonumber \\
&\ge  \frac{s_{n+1}}{n+1}(1+n)- \frac{s_n}{n}n \nonumber\\
&=w_{n+1}. 
 \end{align}
 Combining \eqref{e1} and \eqref{e2} gives us \eqref{e0}.
\end{proof}

\begin{proposition}\label{p:bidemocratic}
Let $\ww$ be a weight. Then the unit vector system of $\DL_\ww$ has Property (A) and  is isometrically bidemocratic.

\end{proposition}

\begin{proof}
The isometric bidemocracy follows from Lemma~\ref{lem:PropNorm}\ref{P:6} and \ref{P:7}.
To obtain Property (A), combine Lemma~\ref{lem:PA} with Lemma~\ref{lem:PropNorm}\ref{P:13}. 
\end{proof}

In order to compute the norm  of specific vectors in $\DL_\ww$, it will be convenient to use the following improvement ot Lemma~\ref{lem:PA}.

\begin{lemma}\label{lem:PB}
Let $\ww$ be a  weight and $f=(a_j)_{j=1}^\infty\in c_{00}$. Then 
\[
\norm{f}_{\DL,\ww}=\sup_{ E_f\subset E \subset\supp(f) } \bphi_1(E;f) + \bphi_2(E;f).
\] 
\end{lemma}

\begin{proof}
It suffices to prove that  the inequality 
\[
 \bphi_1(E;f) + \bphi_2(E;f)\le  \bphi_1(E\setminus\{k\};f) + \bphi_2(E\setminus\{k\};f).
 \]
 holds for all  $E\subset\NN$ with $n:=\abs{E}<\infty$ and all  $k\in E\setminus\supp(f)$. To that end,
pick $\varphi\in\Pi(E)$,  and let  $\psi\in\Pi(E \setminus \{k\}) $ be the map given by
 \[
 \psi(j)= \begin{cases} k & \mbox{ if }  j =n ,  \\  \varphi(j) & \mbox{ if }  j\ge n+1. \end{cases}
 \]
 Since $x_k=0$,
\[
 \sum_{j=n+1}^\infty a_{\varphi(j)} w_{j} = \sum_{j=n}^\infty a_{\psi(j)} w_j.
 \]
Hence,  $\bphi_2(E;f) \le   \bphi_2(E\setminus\{k\};f)$.  Since  $\bphi_1(E;f) \le   \bphi_1(E\setminus\{k\};f)$, we are done.
\end{proof}

\section{The suppression unconditionality constant of bases with Property (A)}\label{s:suppression unconditionality}\noindent
  
\noindent Given a  weight $\ww$, we denote by $\Ks_\ww$ the suppression unconditionality constant of the unit vector system of the Banach space $\DL_\ww$ constructed in Section~\ref{s:building bases}. The following lemma will become instrumental in obtaining a uniform bound for $\Ks_\ww$.

\begin{lemma}\label{lem:TimurAdrenaline}
    Let $\ww$ be a weight, $f\in c_{00}$, and $A \subset\NN$. Then:
    \begin{enumerate}[label=(\roman*),leftmargin=*,widest=ii]
     \item\label{bdd:B} $\ww_\infty \norm{f}_1\le \bphi_1(F;f)$ for some $F\subset \NN$ finite, and
        \item\label{bdd:A} $\bphi_2(E;S_A(f))\le \bphi_2 (E;f) +\ww_\infty\norm{S_{A^c}(f)}_1$ for every $E\subset \NN$ finite.
    \end{enumerate}
\end{lemma}

\begin{proof}
To prove \ref{bdd:B} we choose $F=\supp(f)$ and set $n=\abs{F}$. Then
\[
\bphi_1(F;f)=\frac{s_n}{n} \norm{f}_1\ge w_n  \norm{f}_1\ge \ww_\infty \norm{f}_1.
\]
To see \ref{bdd:A} we pick  $\varphi\in \Pi(E)$. Let us put $n=\abs{E}$,
\[
B=\varphi^{-1}( \supp(f)\cap A), 
\quad\text{and}\quad D=\varphi^{-1} (\supp(f)\setminus A).
\]
If $\varphi^{-1}(\supp(f)) = \emptyset$, i.e., $\supp(f)\subset E$, then $\bphi_2(E;S_A(f)) = 0$ and there is nothing to prove. Assume that $\varphi^{-1}(\supp(f))\neq \emptyset$; we can then let $n_0 = \max\varphi^{-1}(\supp(f))$ and  $m\in\NN_{>n_0}$. Select a  set $D_1\subset \NN_{>m}$ such that $\abs{D_1}=\abs{D}$ and a bijection $\pi:\NN_{>n}\to \NN_{>n}$ such that $\pi(D)=D_1$, $\pi(D_1)=D$, and $\pi(j)=j$ for $j\notin D\cup D_1$. Consider now $\sigma=\varphi\circ\pi$.  Since $D_1\subset \NN_{>m}$  and $\sigma(D_1)=\supp(f)\setminus A$,
\begin{align*}
\abs{ \sum_{j=n+1}^\infty a_{\varphi(j)} \chi_A(\varphi(j)) w_j - \sum_{j=n+1}^\infty a_{\sigma(j)} w_j}&
=\abs{ \sum_{j\in B} a_{\varphi(j)} w_j - \sum_{j\in D_1\cup B} a_{\sigma(j)} w_j}\\
&=\abs{\sum_{ j\in D_1 } a_{\sigma(j)} w_j}\\
&\le w_m \norm{S_{A^c}(f)}_1.
\end{align*}
Applying the triangle law and letting $m$ tend to infinity we obtain 
\[
\abs{\sum_{j=n+1}^\infty a_{\varphi(j)} \chi_A(\varphi(j)) w_j} \le \bphi_2(E;f)+\ww_\infty \norm{S_{A^c}(f)}_1,
\]
and so \ref{bdd:A} holds.
\end{proof}

\begin{proposition}\label{prop:Kwuniformbound}
Let $\ww$ be a  weight. Then $\Ks_\ww\le 2$. Moreover, if $\ww_\infty=0$ then $\Ks_\ww=1$.
\end{proposition}

\begin{proof}
 Pick $f\in c_{00}$ and $A\subset \NN$. Let $F$ be as in Lemma~\ref{lem:TimurAdrenaline}\ref{bdd:B} relative to the vector $S_{A^c}(f)$. Given $E\subset\NN$ finite, Lemma~\ref{lem:TimurAdrenaline}\ref{bdd:A} yields
 \begin{align*}
\bphi(E;S_A(f))&\le \bphi_1(E;S_A(f)) + \bphi_2 (E;f) +\ww_\infty\norm{S_{A^c}(f)}_1\\
 &\le \bphi_1(E;S_A(f)) + \bphi_2 (E;f) +\bphi_1(F;S_{A^c}(f))\\
 &\le \bphi_1(E;f) + \bphi_2 (E;f) +\bphi_1(F;f)\\
 &\le  \bphi(E;f)+ \bphi(F;f).
 \end{align*}
 Moreover, if $\ww_\infty=0$ the term $\bphi(F,f)$ can be dropped.
\end{proof}

Before proceeding further we make a stop en route to give a  nice equivalent norm on $\DL_\ww$.

\begin{proposition}\label{p:compare norms}
If $\ww$ is a  weight then $\DL_\ww = d_1(\ww)$. Quantitatively,
\[
 \norm{f}_{\DL,\ww}\le \norm{f}_{1,\ww} \le C \norm{f}_{\DL,\ww}, \quad f\in c_{00}, 
\]
where $C=4$ if $\FF=\RR$ and  $C=8$ if $\FF=\CC$. Further, in the case when $\ww_\infty = 0$, we can take $C=2$ if $\FF=\RR$ and  $C=4$ if $\FF=\CC$.
\end{proposition}
\begin{proof}
Let $\Upsilon=2$ if $\FF=\RR$ and  $\Upsilon=4$ if $\FF=\CC$. Also, let $K=1$ if $\ww_\infty=0$ and $K=2$ if $\ww_\infty>0$.
The left-hand side inequality follows from Lemma \ref{lem:PropNorm}\ref{P:12}. To see the right-hand side inequality, pick $f\in c_{00}$. By  Lemma~\ref{lem:PropNorm}\ref{P:14},
 \begin{align*}
\norm{f}_{1,\ww}
&\le\norm{\Re^+(f)}_{1,\ww}+\norm{\Re^-(f)}_{1,\ww}+\norm{\Im^+(f)}_{1,\ww}+\norm{\Im^-(f)}_{1,\ww}\\
&= \norm{\Re^+(f)}_{\DL,\ww}+\norm{\Re^-(f)}_{\DL,\ww}+\norm{\Im^+(f)}_{\DL,\ww}+\norm{\Im^-(f)}_{\DL,\ww}\\
&\le K \Upsilon \norm{f}_{\DL,\ww}.\qedhere
 \end{align*}
\end{proof}

In light of Proposition~\ref{prop:Kwuniformbound} in order to find bases that  are not $1$-suppression unconditional despite having Property (A), we must focus on the case where $\ww_\infty>0$. In this situation, by Proposition~\ref{p:compare norms}, $\DL_\ww=d_1(\ww)=\ell_1$ up to an equivalent norm. We could have come to this conclusion without invoking this result. Indeed, 
\[
\ww_\infty \norm{f}_1 \le \norm{f}_{m,\ww} \le \norm{f}_{\DL,\ww} \le \norm{f}_{1,\ww}\le \norm{f}_1
\]
for all $f\in c_{00}$. These estimates for $\norm{\cdot}_{\DL,\ww}$
yield that the unconditionality constant of the unit vector system of $\DL_\ww$ does not exceed $1/\ww_\infty$. Hence, by Proposition~\ref{prop:Kwuniformbound},
\[
\Ks_\ww\le \min\left\{2,\frac{1}{\ww_\infty}\right\}.
\]
Theorem~\ref{thm:Main} will follow once we show that $\Ks_\ww > 1$ for a suitable weight $\ww$.

\begin{proposition}\label{p:renormingclosetotwo}\label{p:renorming closer to optimal}
For each $K<3/2$ there exists a weight $\ww$ so that $\Ks_\ww\ge K$.
\end{proposition}

\begin{proof}
Fix $\omega\in(0,1)$ and consider the weight
\[
\ww = (1, \omega, \omega, \ldots).
\]
We shall estimate $\Ks_\ww$ by comparing the norms of the vectors
\[
f_{n,\omega}=\ee_1+ \omega\Ind_{A_n}, \quad\text{and}\quad  g_{n,\omega}=\ee_1+ \omega\Ind_{A_n}-\omega\Ind_{B_n},
\]
where $A_n=\NN\cap[2,n+1]$ and  $B_n=\NN\cap[n+2,2n+1]$. By Lemma~\ref{lem:PropNorm}\ref{P:14}, 
\[
\norm{f_{n,\omega}}_{\DL,\ww}=1+n\omega^2.
\]
By Lemma \ref{lem:PB} to evaluate $\norm{g_{n,\omega}}$ we need only  consider the case  $\{1\} \subset E \subset \{1\} \cup A_n \cup B_n$. We have
\[
\bphi_1(g_{n,\omega};E)=\frac{(1+\omega(2n-a-b))^2}{2n+1-a-b}, \quad
\bphi_2(g_{n,\omega};E)=\abs{a-b} \omega^2,
\]
where $a=\abs{A_n\setminus  E}$ and  $b=\abs{B_n\setminus  E}$. Hence, for $a+b$ fixed, the maximum value of $\bphi(g_{n,\omega}; E)$  is attained when $a$ attains its minimum value. Therefore, combining Lemma~\ref{lem:PA} with Lemma~\ref{lem:PropNorm}\ref{P:13} gives
\[
\norm{g_{n,\omega}}_{\DL,\ww}=\sup_{0\le k \le 2n} \bphi(g_{n,\omega}; E_k),
\]
where $E_k=[1,k+1]\cap \NN$. Moreover, if $a_k$ and $b_k$ are the integers $a$ and $b$ above defined that correspond to the set $E=E_k$,  then $a_k=n-k$ and $b_k=n$ for all $k=0$, \dots, $n$. By Lemma~\ref{lem:PropNorm}\ref{P:15},
\[
\norm{g_{n,\omega}}_{\DL,\ww}=\sup_{0\le k \le n} \bphi(g_{n,\omega}; E_k)=
\sup_{0\le k \le n}\alpha(\omega,k),
\]
where 
\[
\alpha(\omega,t)=\frac{(1 + \omega t)^2}{t+1} +\omega^2 t, \quad 0\le t \le n.
\]
The expression
\[
\alpha(\omega,t)=\frac{(1-\omega)^2}{1+t}+  2\omega^2 t +2\omega-\omega^2
\]
shows  that the function $\alpha(\omega,\cdot)$ is convex hence the maximum value of $\alpha(\omega,t)$, $0\le t \le n$, is attained either at the endpoint  $t=0$ or the endpoint $t=n$. Therefore,
\[
\norm{g_{n,\omega}}_{\DL,\ww}=\max\enbrace{ 1, \frac{(1+n\omega)^2}{n+1}+n \omega^2}.
\]
From here it follows that 
\[
\Ks_\ww \ge \delta(n,\omega):= \frac{\norm{f_{n,\ww}}_{\DL,\ww}}{\norm{g_{n,\ww}}_{\DL,\ww}}=\min\enbrace{\beta(n,\omega),\gamma(n,\omega)},
\]
where
\[
\beta(n,\omega)=1+n\omega^2, \quad \gamma(n,\omega)=\frac{1+n\omega^2}{n \omega^2 +{(1+n\omega)^2}/(n+1)}.
\]
For $n$ fixed, the function $\beta(n,\omega)$, $0\le \omega \le 1$, increases from $1$ to $1+n$ while the function $\gamma(n,\omega)$, $0\le \omega \le 1$, decreases from 
$n+1$ to $(1+n)/(1+2n)$. Hence, $\delta(n,\omega)$ attains its maximum value at the only point $\omega=\omega_n\in(0,1)$ that satisfies
\[
\beta(n,\omega)=\gamma(n,\omega).
\]
We infer that $\omega_n$ is the positive root of the second order polynomial
$
\omega\mapsto (2n+1)\omega^2 +2\omega-1,
$
that is, 
\[
\omega_n=\frac{-1+\sqrt{2n+2}}{2n+1}=\frac{1}{1+\sqrt{2n+2}}.
\]
Summing up, the weight $\ww=(1,\omega_n,\omega_n,\ldots)$ satisfies
\[
\Ks_\ww\ge K_n:= 1+n \omega_n^2= 1+(1-2\omega_n)\frac{n}{2n+1}.
\]
As $(\omega_n)_{n=1}^\infty$ decreases from $w_1=1/3$ to zero, $(K_n)_{n=1}^\infty$ increases from $K_1=10/9$ to $3/2$.
\end{proof}

\begin{remark}\label{optimality for 2}
Once we know there is a Banach space $\XX$ with a basis $\BB= (\xx_n)_{n=1}^\infty$ having Property (A) and $K_s>1$, one can ask whether this inequality will be fulfilled by vectors with small support. It is easy to check that if  $g \in \XX$ is supported on a set of cardinality  at most $2$, then $\norm{S_A(g)} \le \norm{g}$ for any set $A$. 
This is so because if $g = a \,\xx_k + b\, \xx_n$ with $k \neq n$ and $\abs{a} \leq \abs{b}$, then due to Property (A), $\norm{g} = \norm{ -a \xx_k + b \xx_n}$. Hence, by the triangle law, 
 \[
 \norm{g} \geq \norm{b\, \xx_n} = \abs{b}\norm{\xx_n} \geq \abs{a}\norm{\xx_k}
 =\norm{a\, \xx_k}.
 \]
 
 In contrast, one can witness the inequality $K_s>1$ on vectors supported on three elements. In fact, the proof of Proposition~\ref{p:renorming closer to optimal} gives $g\in D_\ww$ and $A\subset\NN$ with $\abs{\supp(g)}=3$  and $\norm{S_A(g)}/\norm{g}\ge 10/9$. We do not know whether this is optimal. However, it is so if we restrict ourselves to $\DL_\ww$ spaces, and their elements $g=g_a=\ee_1+a\, \ee_2-a\, \ee_3$, $0<a<1$. Indeed, setting $f_a=\ee_1+a\, \ee_2$, we can show that the maximum value of the ratio $\norm{f_a}_{\DL,\ww}/\norm{g_a}_{\DL,\ww}$ when $a$ runs over $(0,1)$ and $\ww$ runs over all possible weights is $10/9$, and its attained when $a=1/3$ and 
$\ww=\enpar{1,{1}/{3},{1}/{3}, {1}/{3},\dots}$.
\end{remark}

The example used to prove Proposition~\ref{p:renorming closer to optimal} is ``tight'' in the sense that $1$-symmetric Markushevich bases which are not equivalent to the standard $\ell_1$-basis 
are $1$-suppression unconditional.  Proposition~\ref{p:can distort l1 only} below substantiates this assertion.
Before stating it, we alert the reader that in spite of the fact that symmetric Schauder bases are always unconditional (see, e.g., \cite{LT1}*{Section 3.a}), there exist  symmetric  bases 
that are not unconditional (see Example~\ref{ex:ConditionalSymmetric} below). A standard argument permits to obtain that a symmetric basis  is $C$-symmetric for some $C\ge 1$ (see, e.g., \cite{AlbiacKalton2016}*{Proof of Lemma 9.2.2}), whence it is $1$-symmetric under a suitable renorming of the space.  We also deduce that any symmetric basis is seminormalized. Since symmetry dualizes well, any symmetric basis satisfies \ref{seminormalized}.
Another important property of symmetric bases that still holds in our more general framework is the boundedness of the averaging projections (see, e.g., \cite{LT1}*{Proposition 3.a.4}).

Given $a\in\FF$, its \emph{sign} will be the number $\sgn(a)=a/\abs{a}\in\EE$, with the convention that $\sgn(0)=1$.

\begin{lemma}\label{lem:NotUncNotUCC}
Suppose $\BB=(\xx_n)_{n=1}^\infty$ is a $1$-symmetric basis of a Banach space $\XX$. If $\BB$ is not unconditional then for every $\delta>0$ there are  $\lambda\in \EE$ and $m\in\NN$ such that
\begin{equation}\label{eq:uccbroken}
N(A,B,\lambda):=\norm{\frac{ \Ind_A}{\norm{\Ind_A}}- \lambda \frac{ \Ind_B}{\norm{\Ind_B}}}\le \delta
\end{equation}
for any disjoint sets $A$, $B \subset \NN$ of finite cardinality of at least $m$.
\end{lemma}

\begin{proof}
Fix $0<\varepsilon<1/2$ so that 
\[
\frac{4\varepsilon(1+\varepsilon)}{(1-\varepsilon)(1-2\varepsilon)} \le \delta .
\]
By the lack of unconditionality 
there are finitely disjointly supported vectors $f$ and $g$
such that $ \norm{f}=1$, $\norm{f+g}\le\varepsilon$. Let $\alpha$ and $\beta$ be the signs of the sums of the coefficients of $f$ and $g$, respectively. Pick $A\supset A_0 := \supp(f)$ and  $B\supset B_0 := \supp(g)$ disjoint and set
\[
f_0=\Ave(f,A) \Ind_A, \quad g_0=\Ave(g,B) \Ind_B.
\]
Since the averaging projections are contractive \cite{LT1}*{Proposition 3.a.4}, $\norm{f+g_0}\le \varepsilon$ and $\norm{f_0+g_0}\le \varepsilon$. The former inequality implies $1-\varepsilon\le \norm{g_0}\le 1+\varepsilon$ while the latter implies that $1-2\varepsilon\le \norm{f_0}\le 1+2\varepsilon$. 
The identity
\[
 \frac{\alpha\Ind_A}{\norm{\Ind_A}}+  \frac{\beta \Ind_B}{\norm{\Ind_B}}
 = \frac{f_0}{ \norm{f_0}}  + \frac{g_0}{\norm{g_0}} = 
\frac{ \norm{g_0}-\norm{f_0}}{ \norm{f_0} \norm{g_0}} f_0 + \frac{f_0 +g_0}{ \norm{g_0}} 
\]
yields the inequality
\[
\norm{ \alpha \frac{\Ind_A}{\norm{\Ind_A}} +\beta  \frac{ \Ind_B}{\norm{\Ind_B}}}\le  \frac{3\varepsilon(1+2\varepsilon)}{(1-2\varepsilon)(1-\varepsilon)} + \frac{\varepsilon}{1-\varepsilon} = \frac{4\varepsilon(1+\varepsilon)}{(1-\varepsilon)(1-2\varepsilon)} \le \delta,
\] 
so that  $A$ and $B$ satisfy \eqref{eq:uccbroken} with $\lambda=-\alpha^{-1}\beta$. Since the basis is $1$-symmetric, $N(A,B,\lambda)$ only depends on $\abs{A}$, $\abs{B}$, and $\lambda$  provided that $A$ and $B$ are disjoint. Hence \eqref{eq:uccbroken} still holds for $A$ and $B$ disjoint subsets of $\NN$ of finite cardinality of at least $m:=\max\{\abs{A_0},\abs{B_0}\}$.
\end{proof}

\begin{proposition}\label{p:can distort l1 only}
Suppose $\BB=(\xx_n)_{n=1}^\infty$ is a $1$-symmetric total basis in a Banach space $\XX$ which fails to be $1$-suppression unconditional. 
Then $\BB$ is equivalent to the canonical $\ell_1$-basis. 
\end{proposition}

\begin{proof}
Suppose $\BB$ is a $1$-symmetric basis which is not $1$-suppression unconditional. 
Note first that $\BB$ cannot be weakly null. Indeed, suppose it were. As $\BB$ fails to be $1$-suppression unconditional, we can find finitely disjointly supported vectors $x$ and $y$  such that $\norm{x}>\norm{x+y}$. For each $p\in\NN$, let $y_p$ be the $p$-th right-shift of $y$, i.e., 
\[
y_p=\sum_{n=1}^\infty \xx_n^*(y) \xx_{n+p}.
\]
Since $y_p$ is weakly null, $\lim_p (x+y_p)=x$ weakly. Therefore,
\[
\norm{x} \le \limsup_p  \norm{x+y_p}.
\]
For $p$ large enough $x$ and $y_p$  are disjointly supported, whence, by symmetry,
$
\norm{x+y}= \norm{x+y_p}
$. Summing up, we obtain
\[
\norm{x} \le \norm{x+y},
\]
 a contradiction.

Thus, there exist a subsequence $\BB_0:=(\xx_{n_i})_{i=1}^\infty$ of $\BB$, a linear functional $x^*\in\XX^*$ with $\norm{x^*}=1$, and $\alpha>0$ such that
\[
\Re( x^*(\xx_{n_i}))\ge \alpha, \quad i\in\NN.
\]
Then for any sequence $(a_i)_{i=1}^\infty$ of nonnegative scalars, we have
\begin{equation}
\label{eq:positive l1}
\norm{ \sum_{i=1}^\infty a_i \xx_{n_i} } \geq \Re\enpar{ x^* \enpar{ \sum_{i=1}^\infty a_i \xx_{n_i}}} \geq \alpha \sum_{i=1}^\infty a_i.
\end{equation}

By way of contradiction, assume that $\BB$ is not an unconditional basis. 
Pick $(\delta_k)_{k=1}^\infty$ in $(0,\infty)$ with $\sum_{k=1}^\infty \delta_k<\infty$.
 For each $k\in\NN$, let $m_k=m(\delta_k)\in\NN$ and  $\lambda_k=\lambda(\delta_k)\in\EE$ be the numbers provided by  Lemma~\ref{lem:NotUncNotUCC}. Set $m_0=1$ and pick a sequence $(A_k)_{k=1}^\infty$ consisting of pairwise disjoint subsets of $\NN$ with  $\max\{m_{k},m_{k-1}\}  \le \abs{A_k}<\infty$. We then have
\[
\norm{\frac{ \Ind_{A_k}}{\norm{\Ind_{A_k}}}
-\lambda_k\frac{\Ind_{A_{k+1}}}{\norm{\Ind_{A_{k+1}}}}}\le \delta_k, \quad k\in\NN.
\]
It follows that the sequence $(h_k)_{k=1}^\infty$ defined by
\[
h_k= \enpar{\prod_{i=1}^{k-1} \lambda_i}  \frac{\Ind_{A_k}}{\norm{\Ind_{A_k}}}, \quad k\in\NN,
\]
 is Cauchy and does not converge to zero. Since $\lim_k \xx_n^*(h_k)=0$ for all $n\in\NN$, $\BB$ is not a total basis. This contradiction shows that $\BB$ is  unconditional. Therefore, by \eqref{eq:positive l1},  $\BB_0$ is equivalent to the canonical $\ell_1$-basis. Since unconditional symmetric bases are subsymmetric (see \cite{LT1}*{Proposition 3.a.3}), $\BB$ is also equivalent to the canonical $\ell_1$-basis.
\end{proof}

We rely on our analysis of symmetric bases (not necessarily Schauder) to generalize the aforementioned classical result that symmetric Schauder bases are unconditional. 

\begin{corollary}\label{cor:symunc}
Given a symmetric basis $\BB$ of a Banach space $\XX$, the following are equivalent
\begin{enumerate}[label=(\roman*), leftmargin=*]
    \item\label{symunc:a} $\BB$ is an unconditional basis.
    \item\label{symunc:d} $\BB$ is bidemocratic.
    \item\label{symunc:e} $\BB$ is super-democratic.
    \item\label{symunc:b} $\BB$ is unconditional for constant coefficients.
    \item\label{symunc:c} $\BB$ is a Markushevich  basis.
\end{enumerate}
\end{corollary}

\begin{proof}
Using \cite{LT1}*{Proposition 3.a.6}), \ref{symunc:a} implies \ref{symunc:d}. 
By \cite{AABW}*{Section 5}, \ref{symunc:d} implies \ref{symunc:e}, and  \ref{symunc:e} implies \ref{symunc:b}. To establish \ref{symunc:b} $\Rightarrow$ \ref{symunc:a}, suppose \ref{symunc:a} fails. By renorming, we can assume that $\BB$ is $1$-symmetric. 
 Given $\delta>0$, use Lemma~\ref{lem:NotUncNotUCC} to pick  $\lambda\in \EE$, and $A$ and $B$ disjoint subsets of $\NN$ with $ \abs{A}=\abs{B}<\infty$ and $N(A,B,\lambda)\le \delta$. Since, by $1$-symmetry, $\norm{\Ind_A}=\norm{\Ind_B}$, 
\[
\norm{\Ind_A -\lambda\Ind_B} \le  \delta\norm{\Ind_A}.
\]
Therefore, \ref{symunc:b}  is false as well.

Clearly, \ref{symunc:a} implies \ref{symunc:c}.  Finally, if \ref{symunc:c} holds, by Proposition~\ref{p:can distort l1 only}, $\BB$ is either $1$-suppression unconditional or equivalent to the unit vector system of $\ell_1$. In both cases \ref{symunc:a} holds.
\end{proof}

\begin{example}\label{ex:ConditionalSymmetric}
Here we provide an example of a $1$-symmetric basis  which is not unconditional (hence not total, by Corollary \ref{cor:symunc}). 

Let $\ww=(w_n)_{n=1}^\infty$ be a weight.  For $f=(a_n)_{n=1}^\infty\in c_{00}$, put
\begin{equation}
    \label{eq:symmetric}
\norm{f}=\bphi_2(\emptyset;f)=\sup_{\varphi\in \Pi} \abs{\sum_{j=1}^\infty a_{\varphi(j)} w_j}.
\end{equation}
Clearly, $\norm{\cdot}$ is a semi-norm that satisfies $\max\{ \norm{\Re(f)}, \norm{\Im(f)}\} \le \norm{f}$ for all $f\in c_{00}$. Moreover, if $f=(a_j)_{j=1}^\infty$ is real-valued,
\[
\norm{f}=\max\enbrace{ \Psi(f),  \Psi(-f)},
\]
where
\[
\Psi(f)=\sum_{n=1}^\infty  b_j^+ w_j -\ww_\infty \sum_{n=1}^\infty  b_j^-,
\]
and $(b_j^+)_{j=1}^\infty$ and $(b_j^-)_{j=1}^\infty$ denote the nonincreasing rearrangements of $f^+$ and $f^-$, respectively. Then,
\[
\Theta(f):=\sum_{j=1}^\infty (b_j^+ + b_j^-)(w_j-\ww_\infty)=\Psi(f)+\Psi(-f) \le 2 \norm{f}.
\]
This way, if $\ww$ is not constant, 
\[
\abs{a_m}\le b_1^++b_1^- \le \frac{\Theta(f)}{1-\ww_\infty}\le \frac{2}{1-\ww_\infty} \norm{f}, \quad m\in\NN.
\]
As a consequence, even if $f=(a_j)_{j=1}^\infty$ is not real-valued, 
\begin{equation}
\abs{a_m}  \le \frac{2}{1-\ww_\infty} \norm{\Re  \enpar{ \frac{f}{\sgn(a_m)}}}\   \le \frac{2}{1-\ww_\infty} \norm{f}, \quad m\in\NN.
\label{norm of dual}
\end{equation}
We deduce that  $\norm{\cdot}$ is a norm on $c_{00}$ whose completion is a Banach space, say $\XX$, for which the unit vector system (which we denote by $(\ee_j)_{j=1}^\infty = \EB$) is $1$-symmetric. 
From the definition of $\| \cdot \|$ we see that $\|\ee_m\| = 1$ for all $m$, while \eqref{norm of dual} shows that the biorthogonal functionals satisfy $\|\ee_m^*\| \leq 2/(1-\ww_\infty)$.

If $\ww_\infty=0$, 
\[
\frac{1}{2}\norm{f}_{1,\ww}\le \frac{1}{2} \Theta(f)\le \norm{f} \le\norm{f}_{1,\ww}
\]
for all real-valued $f\in c_{00}$.
This gives us a renorming of $d_{1,\ww}$.  

Let us show that if $\ww_\infty>0$, $\EB$ is not unconditional for constant coefficients (UCC for short).  Let $(s_n)_{n=1}^\infty$ be the primitive weight of $\ww$. Set  $\varepsilon=((-1)^n)_{n=1}^\infty$ and, for each $m\in\NN$, $A_m=\{j\in\NN \colon j \le 2m\}$. If  $\EB$ were UCC, then there would be a constant $C\in[1,\infty)$ such that
\[
s_{2m}= \norm{\Ind_{A_m}} \le C \norm{\Ind_{\varepsilon,A_m}}=C ( s_m - m\ww_\infty).
\]
Hence, if we set $t_n=s_n/n$ for all $n\in\NN$,
\[
2 t_{2m} \le C( t_m-\ww_\infty), \quad m\in\NN.
\]
Since $\lim_m t_m=\ww_\infty$, we reach the absurdity $2\ww_\infty\le 0$.
\end{example}

The study of the lattice-unconditional constants of  bases with  Property (A) is also of interest. Any $1$-greedy basis is $1$-suppression unconditional, hence $2$-lattice unconditional for real Banach spaces. Conversely, \cite{DOSZ}*{Theorem 4.1} shows that the constant $2$ is optimal, i.e., there is a $1$-greedy basis which is not $C$-lattice-unconditional for any $C\le 2$.  The basis constructed in the proof of Proposition \ref{p:renorming closer to optimal} is $2$-suppression unconditional by Proposition \ref{p:compare norms}, hence $4$-unconditional.

On the other hand, taking $\ww = (1, w_n, w_n, \ldots)$ and comparing the norms in $\DL_n:=\DL_{\ww}$ of $g_n = \ee_1 + w_n( \Ind_{A_n} - \Ind_{B_n})$ and $h_n = \ee_1 + w_n( \Ind_{A_n} + \Ind_{B_n})$, where $w_n$, $A_n$ and $B_n$ are as in the proof of Proposition~\ref{p:renormingclosetotwo}, we see that the unconditionality constant of the unit vector system of $\DL_n$ is at least
\[
\frac{\norm{h_n}_{\DL,\ww}}{\norm{g_n}_{\DL,\ww}}=1+2nw_n^2= 1+(1-2\omega_n)\frac{2n}{2n+1} \underset{n \to \infty}{\longrightarrow} 2^-.
\]
Therefore, for every $s < 2$ there exists a sufficiently large $n$ such that the unconditionality constant of the canonical basis of $\DL_n$ is greater than $s$. We do not know whether $2$ or the upper estimate for the unconditionality constant we have achieved (that is, $4$) is closer to its actual value.

\begin{question}\label{Q4} Does an unconditional basis with Property (A) always have $K_l \le 2$? Does it have $K_s\le 2$? Does it have $K_s\le D$ for some universal constant $D$?
\end{question}



\section{Remarks on Property (A) in the finite dimensional case}\label{s:fin dim}\noindent

\noindent Property (A), like any other greedy-like property for that matter, can be defined for bases in finite dimensional spaces with the obvious adjustments. 

If a basis has Property (A), then any (finite or infinite) subbasis inherits this feature. However, unlike other greedy-like properties, Property (A) does not pass to direct sums even when the fundamental functions of the summands are the same. Therefore, it is by no means clear how an infinite basis with Property (A) can be constructed from finite bases  with Property (A). In this regard, it is worth mentioning that we can 
infer finite dimensional results from the infinite dimensional ones.

\begin{proposition}\label{p:PAFinite}
For each $d\in\NN$, $d\ge 2$, there is a $d$-dimensional Banach space with a basis  that satisfies Property (A) and has suppression unconditionality constant $K_s \ge 10/9$.
\end{proposition}

Indeed, for $d\ge 3$ we consider $\RR^d$ with the norm
\[
\norm{x}_d=\norm{J_d(x)}_{\DL,\ww}, \quad x\in\RR^d,
\]
where $\ww = (1, 1/3, 1/3, \ldots)$ 
and $J_d\colon\RR^d \to c_{00}$ is the natural embedding.
In this situation, by Lemma~\ref{lem:PB},
\[ \norm{(a_j)_{j=1}^d}_d= \max\enbrace{ \frac{n+2}{3n}\sum_{j=1}^n \abs{a_{\pi(j)} } +\frac{1}{3} \abs{\sum_{j=n+1}^d a_{\pi(j)} } },\]
where the maximum is taken over $n=1$, \dots, $d$ and all permutations $\pi$ of $\{1, \ldots, d\}$.
The proof of Proposition \ref{p:renorming closer to optimal} shows that $K_s \geq 10/9$.

When $d=2$, Property (A) gives us no information (beyond normalization); hence we do not expect to attain $1$-suppression unconditionality. For example, fix $\alpha \in (0,1)$ and set
\[
\norm{ a_1 \ee_1 + a_2 \ee_2 }= \max \enbrace{ \alpha \abs{a_1}, \alpha \abs{a_2}, \abs{a_1 + a_2} }.
\]
The unit ball of the resulting normed space is the $6$-gon with vertices $\pm (-\alpha^{-1}, \alpha^{-1})$, $\pm (1-\alpha^{-1}, \alpha^{-1})$, and $\pm (\alpha^{-1}, 1-\alpha^{-1})$. Then 
\[\|\alpha^{-1} (\ee_1 - \ee_2)\| = 1,\quad \text{but}\quad \|\alpha^{-1} \ee_1\| = \alpha^{-1},
\]
 so the suppression unconditionality constant is at least $\alpha^{-1}$.
 
 Finally, we show that in the finite-dimensional setting Property (A) does not pass to the dual basis. 
 We do not know whether this still holds in infinite-dimensional spaces.

\begin{proposition} \label{bad dual}
 For $d \geq 3$ the space $\RR^d$ admits a norm $\norm{\cdot}$ so that:
 \begin{enumerate}[leftmargin=*]
  \item The unit vector system is $\EB=(\ee_j)_{j=1}^d$ and its dual basis $\EB^*=(\ee_j^*)_{j=1}^d$ are normalized and $1$-symmetric.
  \item The basis $\EB$ has Property (A).
  \item The dual basis $\EB^*$ fails Property (A). 
 \end{enumerate}
\end{proposition}

In fact the proof shows that for $A = \{1, \ldots, d-1\}$ there exist $\varepsilon, \delta \in \EE^A$ so that $\norm{ \Ind^*_{\varepsilon,A} } \neq \norm{ \Ind^*_{\delta,A} }$.

\begin{proof}
 For $(a_j)_{j=1}^d \in\RR^d$ we define
 \begin{equation}
\norm{(a_j)_{j=1}^d} =\norm{ \sum_{j=1}^d  a_j \, \ee_j } := \sup_{\pi \in \Pi_d} \enbrace{\abs{ a_{\pi(1)} + \frac{a_{\pi(d)}}3 } + \sum_{j=2}^{d-1} \abs{a_{\pi(j)}}},
 \label{eq:define norm}
 \end{equation}
 where  $\Pi_d$ denotes the set consisting of all permutations of $\{1, \ldots, d\}$. Clearly both $\EB$ and $\EB^*$ are $1$-symmetric. Moreover, 
 \[
 \sup_j \abs{a_j} \leq \norm{ \sum_{j=1}^d a_j\, \ee_j } \leq \sum_{j=1}^d \abs{a_j},
 \]
 so the bases $\EB$ and $\EB^*$ are normalized. If some of the coefficients $a_j$ are equal to $0$ then $\norm{\sum_{j=1}^d a_j \, \ee_j } = \sum_{j=1}^d \abs{a_j}$ and the basis $\EB$ has Property (A).  We shall complete the proof by showing that the functionals
 \[
 h^*= \ee_1^* - \sum_{i=2}^{d-1} \ee_j^*, \quad g^*= \sum_{j=1}^{d-1} \ee_j^*
 \] 
 satisfy $\norm{h^*} \le 1 <\norm{g^*}$. This will prove that the basis $\EB^*$ is not $1$-superdemocratic and so it fails Property (A). Here, a basis is said to be $1$-superdemocratic if $\norm{\Ind^*_{\varepsilon,A}} = \norm{ \Ind^*_{\delta,B} }$ for all $\varepsilon, \delta\in \mathbb{E}$ and for all $A, B\subset\mathbb{N}$ with $|A| = |B|$.
 
 To estimate the norm of $h^*$, note that 
\[
\norm{f}=\sup_{f^*\in\Ft} \abs{f^*(f)}, \quad f\in\RR^d,
\]
where $\Ft$ consists of all functionals of the form
\[
\varepsilon_1 \left( \ee_{\pi(1)}^* + \frac{\ee_{\pi(d)}^*}3 \right) + \sum_{j=2}^{d-1} \varepsilon_j \, \ee_{\pi(j)}^* , \quad \varepsilon=(\varepsilon_j)_{j=1}^{d-1}\in \EE^{d-1}, \, \pi\in\Pi_d.
\]
Therefore, $\norm{f^*}\le 1$ for all $f^*\in\Ft$. The identity
\begin{equation*}
h^*= \frac12 \left( \left( \ee_1^*+  \frac{\ee_d^*}3 - \sum_{j=2}^{d-1} \ee_j^* \right) + \left(-\ee_{d-1}^* - \frac{\ee_d^*}3 +\ee_1^* - \sum_{j=2}^{d-2} \ee_j^*  \right) \right),
\end{equation*}
tells us that the functional $h^*$ belongs to the convex hull of $\Ft$ and so $\norm{h^*}\le 1$.
To estimate the norm of $g^*$, we test it on
\[
g= \sum_{j=1}^{d-1} \ee_j - \frac{\ee_d}2 .
\]
As $g^*(g) = d-1$, it suffices to prove that $\norm{g} < d-1$. Given $\pi\in\Pi_d$, we  compute the number on the right side of \eqref{eq:define norm} depending on whether $\pi(d)=d$ or not. In the former case, the computation gives
\[
\abs{ 1 - \frac16 } + (d-2) = d-1-\frac16,
\]
while in the latter case it produces
\[
\abs{ 1 + \frac13 } + (d-3) + \frac12 = d-1-\frac16 .
\]
Thus $\norm{g} = d-7/6 < d-1$ as desired.
\end{proof}


\end{document}